\documentclass[11pt]{amsart}
\usepackage{amsmath}
\usepackage{graphicx} 
\usepackage{amsmath, amsthm, amssymb, amsfonts}

\title{Measure comparison problems for dilations of convex bodies}
\author{Malak Lafi 
 and Artem Zvavitch}
 \thanks{Both authors are supported in part by the U.S. National Science Foundation Grant DMS-1101636 and
the United States - Israel Binational Science Foundation (BSF) Grant 2018115} 


\date{}

\begin{document}

\newcommand{\vol}{\mathrm{vol}}

\newtheorem{defi}{Definition}
\newtheorem{ques}{Question}

\newtheorem{lemma}{Lemma}
\newtheorem{theorem}{Theorem}
\newtheorem{claim}{Claim}
\newtheorem{remark}{Remark}
\newtheorem{fact}{Fact}

\begin{abstract}We study a version of the Busemann-Petty problem for $\log$-concave measures with an additional assumption on the dilates of convex, symmetric bodies. One of our main tools is an analog of the classical large deviation principle applied to  
$\log$-concave measures, depending on the norm of a convex body. We hope this will be of independent interest.

\end{abstract}

\maketitle

\section{Introduction}

We denote by ${\mathbb R}^n$ the $n$-dimensional Euclidean space equipped with the inner product $\langle\cdot,\cdot\rangle$ and the standard orthonormal basis $\{e_1,\dots,e_n\}$, and denote by $|\cdot|$ the standard Euclidean norm on ${\mathbb R}^n$. For a measurable set $A\subset {\mathbb R}^n$ we refer to its volume (the Lebesgue measure) by $|A|,$ and its boundary by $\partial A.$ The notation $B_2^n$ stands for the closed unit ball in ${\mathbb R}^n,$ and ${\mathbb S}^{n-1}$ for the unit  sphere, i.e. ${\mathbb S}^{n-1}=\partial B_2^n$. A convex body is a convex, compact set with a nonempty interior. Furthermore, a convex body $K$ is symmetric if $K=-K$. A measure $\mu$ is $\log$-concave on ${\mathbb R}^n$ if for every pair of non-empty compact sets   $A$ and $B$ in $\mathbb{R}^n$ and $0<\lambda<1$ we have 
$$\mu (\lambda A+(1-\lambda)B)\ge \mu(A)^\lambda  \mu(B)^{1-\lambda},
$$
where the addition is the Minkowski sum which is defined as the set $A+B=\{a+b: a\in A, b\in B\},$ and the constant multiple (dilation) of a set $A \subset {\mathbb R}^n$ and $\alpha \in {\mathbb R}$ is defined as $\alpha A=\{\alpha a: a\in A\}$.   It follows from the Prékopa-Leindler inequality \cite{Prékopa1, Prékopa2, GardnerBM}, that if a measure $\mu$ that is defined on the measurable subsets of $\mathbb{R}^n$, is generated by a $\log$-concave density, then  $\mu$ is also $\log$-concave. Furthermore, Borell provides a characterization for $\log$-concave measures \cite{Borell}, precisely a locally finite and regular Borel measure $\mu$ is $\log$-concave, if and only if, its density (with respect to the Lebesgue on the
appropriate subspace) is $\log$-concave.

In 1956, Busemann and Petty \cite{BP} posed the following volume comparison problem:
Let $K$ and $L$ be symmetric convex bodies in $\mathbb{R}^n$ so that the $(n-1)$-dimensional volume of every central hyperplane section of $K$ is smaller than the same section for $L$. Does it follow that the $n$-dimensional volume of $K$ is smaller than the $n$-dimensional volume of $L$? In the late 90's, the Busemann-Petty problem was solved as a result of many works \cite{ Gardner, GKS, Kold5, Lutwak, Zhang2}. The answer is affirmative when $n\le 4,$ and negative whenever $n\ge 5$.   It is natural to consider an analog of the Busemann-Petty problem for a more general class of measures.  The first result in this direction was a solution of the Gaussian analog of the Busemann-Petty problem  \cite{Zvavitch1}. It turns out that the answer for the Busemann-Petty problem is the same if we replace the volume with the Gaussian measure. Moreover, it was proved in \cite{Zvavitch2} that the answer is the same if we replace the volume with any measure with continuous, positive, and even density.

V. Milman \cite{Vmilm} asked whether the answer to the Gaussian Busemann-Petty problem would change in a positive direction if we compared not only the Gaussian measure of sections of the bodies but also the Gaussian measure of sections of their dilates, that is, consider two convex symmetric bodies $K, L \subset\mathbb{R}^n $, such that 
\begin{equation*}
    \gamma_{n-1}(rK\cap \xi^\perp)\le \gamma_{n-1}(rL\cap \xi^\perp),  \hspace{0.5cm} \forall \xi\in\mathbb{S}^{n-1}, \hspace{0.5cm}\forall r>0,
\end{equation*}
 where $\xi^\perp$ denotes the central hyperplane perpendicular to $\xi.$ Does it follow that $$\gamma_n(K)\le\gamma_n(L)?$$
Here $\gamma_{n}$ denotes the standard Gaussian measure on $\mathbb{R}^n$, and $$\gamma_{n-1}(K\cap \xi^\perp)=\frac{1}{(\sqrt{2\pi} )^{n-1}} \int_{K\cap \xi^\perp} e^{-\frac{|x|^2}{2}} dx.$$
The addition of dilation to the Busemann-Petty problem, clearly, would not change anything in the case of the volume measure. Still in the case of more general $\log$-concave measures the behavior of the measure of a dilation of a convex body is very interesting, we refer to \cite{BN, CFM, F, LO} for just a few examples of such results.

Even though the dilation adds some strength to the condition of the bodies, the answer to the dilation problem for Gaussian measure is positive for $n\le 4$ and negative for $n\ge7$ (see \cite{Zvavitch3}). That leaves the problem open for $n=5, 6.$ 

To show the strength of the condition of the dilates,  it was proved in \cite{Zvavitch3} that the dilation problem has an aﬃrmative answer when $K$ is a dilate of a centered Euclidean ball: Consider a star body $L \subset\mathbb{R}^n$ and assume there exists $R>0,$ such that
$$\gamma_{n-1}(rRB_2^n\cap \xi^\perp)\le \gamma_{n-1}(rL\cap \xi^\perp),     \hspace{.5cm} \forall \xi\in\mathbb{S}^{n-1},\hspace{.5cm} \forall r>0$$ then it follows that $RB_2^n\subseteq L.$

In this paper, we review some generalizations of the above fact. In particular,  we study measures $\mu$ for which we  have an affirmative answer for the following problem:

\begin{ques}\label{ques1}
 Consider a convex, symmetric body $K \subset {\mathbb R}^n$ such that for every $t$ large enough and for some $R>0$
$$\mu(tRB_2^n) \le \mu (tK),$$
does it follow that $RB_2^n \subseteq K ?$
\end{ques}

In Section \ref{secRI}, we will present a solution for Question \ref{ques1} for the case of a  $\log$-concave, rotation invariant probability measure $\mu.$

In Section \ref{genL}, we consider a more general case, instead of comparing $K$ with $B_2^n$, we will compare $K$ with another convex, symmetric body $L$. Let us denote by   $\| x\|_{L}$  the Minkowski functional of $L$ which is defined to be $\| x\|_{L} = \min \{\lambda > 0: x \in \lambda L\}.$ 

\begin{ques}\label{ques2}  Let $K, L\subset \mathbb{R}^n$ be convex,   symmetric bodies, and let $\mu$ be a $\log$-concave probability measure, with density $e^{-\phi(\|x\|_{L})}$, where $\phi: [0,\infty) \to [0, \infty)$ is an increasing convex function. If for every $t$ large enough and some $R>0$
$$
\mu(tRL) \le \mu (tK),
$$
does it follow that  $RL \subseteq K$?
\end{ques}
One of the core steps in answering Questions \ref{ques1} and \ref{ques2} is a generalization of the classical large deviation principle, which is provided in Lemma \ref{lm2} and equation (\ref{ldL}) below: Consider two  symmetric, convex bodies $K, L \subset {\mathbb R}^n$, let $r(K, L)=\max\{R>0: RL\subset K\}$, then 
$$
\limsup_{t\to\infty}\frac{\ln{\mu((tK)^c)}}{\phi(r(K,L)t)}=-1,
$$
where $\mu$ is a $\log$-concave probability measure, with density $e^{-\phi(\|x\|_{L})}$ and  by $A^c$ we denote a complement of a set $A \subset {\mathbb R}^n,$ i.e.
$A^c =\mathbb{R}^n\setminus A$.

Finally, in Section \ref{BPdilation}, we will  discuss the generalization of the dilation problem for Gaussian measures:

\begin{ques}\label{BPr} Consider a measure $\mu$ with continuous positive density $f$. Let  $\mu_{n-1}(K\cap \xi^\perp)=\int_{K\cap \xi^\perp}f(x)dx.$ Consider  two convex symmetric bodies $K, L \subset\mathbb{R}^n $, such that 
$$
    \mu_{n-1}(rK\cap \xi^\perp)\le \mu_{n-1}(rL\cap \xi^\perp),  \hspace{0.5cm} \forall \xi\in\mathbb{S}^{n-1}, \hspace{0.5cm}\forall r>0,
$$
does it follow that 
$$
\mu(K)\le \mu(L)?
$$
\end{ques}
We show that in general, the answer is still negative in dimension $n\ge 5$, even under the assumption that $f$ is  non-constant $\log$-concave function. We also prove that if we add the requirement for the measure to be rotation invariant, the answer will be negative in dimension $n\ge 7$, which leaves the case of rotation invariant $\log$-concave measures open in dimension $n=5,6.$

{\bf Acknowledgments.}
We are grateful to Matthieu Fradelizi, Dylan Langharst, Fedor Nazarov and  Mokshay Madiman for a number of valuable discussions and suggestions.  Finally, we thank the two anonymous referees,
whose remarks and corrections were an enormous help!

\section{The case of rotation invariant measures}\label{secRI}
In this section we  consider a rotation invariant probability $\log$-concave measure $\mu$ with non-constant density, i.e.
$$
\mu(A)=\int_{A}e^{-\phi(|x|)}dx,
$$
where $\phi: [0,\infty) \to [0, \infty)$ is an increasing convex function. We will denote by $\phi'(t)$ the left derivative, in the case when the convex function $\phi(t)$ is not differentiable at $t$.

\begin{theorem}\label{thm1} Consider a convex, symmetric body $K \subset {\mathbb R}^n$ such that for every $t$ large enough and some $R>0$
$$
\mu(tRB_2^n) \le \mu (tK),
$$
then $RB_2^n \subseteq K.$
\end{theorem}

In order to prove the above theorem, we will need two lemmas.
\begin{lemma} \label{lm1}
Consider $R > 0,$ then
\begin{equation}\label{inlimsup}
\limsup_{t\to \infty}\frac{\ln{\mu((tRB_2^n)^c)}}{\phi(tR)}=-1.
\end{equation}

\end{lemma}
\begin{proof} 
Without loss of generality, we may assume that 
$R=1$. Let us first show that the left-hand side of equality (\ref{inlimsup}) is less or equal to $-1$.
Writing the integral in polar coordinates, we get
$$ \limsup_{t\to\infty}\frac{\ln{\int_{S^{n-1}}\int_t^{\infty} 
e^{-\phi(r)} r^{n-1}drd\theta}}{\phi(t)} =\limsup_{t\to\infty}\frac{\ln{\int_t^{\infty} 
e^{-\phi(r)} r^{n-1}dr}}{\phi(t)},$$ we remind that $\lim_{t \to \infty} \phi(t)=\infty.$
Let $\eta(t)=-(n-1) \ln{t}+\phi(t)$. 
Using that $\phi$ is a convex and non-constant function, we get that there exists $t_0 \ge 0$ such that $ \phi'(t_0)>0$. Thus $\phi'(t)>0,$ for  all  $t>t_0$ and there exists constant $a>0$ such that $\eta'(t) >a,$ for all $t>t_0$. 
Thus, $\eta(r)\geq\eta(t)+ a(r-t)$, for $r>t>t_0$, and \begin{align*}
 \limsup_{t\to\infty}\frac{\ln{\int_t^{\infty} 
e^{-\phi(r)} r^{n-1}dr}}{\phi(t)} & \leq \limsup_{t\to\infty}\frac{\ln{\int_t^{\infty} e^{-\eta(t)-a(r-t)} dr}}{\phi(t)} \\ &=\limsup_{t\to\infty}\frac{\ln{e^{-\eta(t)}}+\ln{\int_t^{\infty} 
e^{-a(r-t)} dr}}{\phi(t)} \\ &=\limsup_{t\to\infty}\frac{\ln(t^{n-1}e^{-\phi(t)})+\ln{\frac{1}{a}}}{\phi(t)} \\ & =- 1 +\limsup_{t\to\infty}\frac{\ln \frac{1}{a} }{\phi(t)}\\
&=-1.
\end{align*}
Next, we will show that the right-hand side of equality (\ref{inlimsup}) is greater or equal to $-1$. Since $r>t$ we have
\begin{align*}
\limsup_{t\to\infty}\frac{\ln{\int_t^{\infty} 
e^{-\phi(r)} r^{n-1}dr}}{\phi(t)} & \geq \limsup_{t\to\infty}\frac{\ln\left(t^{n-1}\int_t^{\infty} 
e^{-\phi(r)} dr\right)}{\phi(t)}
\\& =\limsup_{t\to\infty}\frac{\ln{\int_t^{\infty} 
e^{-\phi(r)} dr}}{\phi(t)}.
\end{align*}
To finish proving the lemma, we prove the following claim.
\begin{claim}\label{claim1}
$\limsup_{t\to\infty}\frac{\ln{\int_t^{\infty} 
e^{-\phi(r)} dr}}{\phi(t)}\geq{-1}.$
\end{claim}
\noindent{\it Proof of Claim \ref{claim1}.} Assume the result is not true, then there exists $\alpha >1$ such that 
$$
\limsup_{t\to\infty}\frac{\ln{\int_t^{\infty} 
e^{-\phi(r)} dr}}{\phi(t)}<{-\alpha}.
$$
Thus, there exists $t_0>0$ such that for all $t >t_0$ we have, 
\begin{equation}\label{cont1}
\int_t^{\infty}e^{-\phi(r)}dr\leq e^{-\alpha\phi(t)}.
\end{equation}
Let $F(t)=\int_t ^{\infty}e^{-\phi(r)} dr$, note that $F^\prime (t)=-e^{-\phi(t)}$ and thus (\ref{cont1}) is equivalent to  
$
F(t)^{\frac{1}{\alpha}}\le -F^\prime (t).
$
Therefore for  $t>t_0$ we have 
$$
1\le-\frac{F^\prime (t)}{F(t)^{\frac{1}{\alpha}}}.
$$
 Integrating both sides of the above inequality over $t \in [t_0, \infty)$, we get that  $\frac{1}{1-\frac{1}{\alpha}} F(t)^{1-\frac{1}{\alpha}}$ is unbounded,  which gives a contradiction and the claim is proved. This finishes the proof of Lemma \ref{lm1}.
 \end{proof}
  \begin{remark}\label{remarkmeasure} We note that in Claim \ref{claim1} we have proved a stronger statement. Indeed, fix $\alpha >1$ and let
 $$
 E=\left\{t: \ln{\int_t^{\infty} 
e^{-\phi(r)} dr} < -\alpha\phi(t)\right\}
 $$
 then $|E|<\infty$. This follows from the fact that, again, for all $t\in E$ we have that 
$$1<-\frac{F^\prime (t)}{F(t)^{\frac{1}{\alpha}}},$$
thus
\begin{equation}\label{lessinfty}
|E| \le \int_E -\frac{F^\prime (t)}{F(t)^{\frac{1}{\alpha}}} dt \le \int_{t_0}^\infty -\frac{F^\prime (t)}{F(t)^{\frac{1}{\alpha}}} dt <\infty.
\end{equation}
 \end{remark}

\begin{remark}
It is tempting to replace limit superior by the actual limit in the statement of Lemma \ref{lm1}. This may be done in many particular cases of measure $\mu$, but it is not true in general. Indeed, if we assume that 
$$\lim_{t\to\infty}\frac{\ln{\int_t^{\infty} 
e^{-\phi(r)} dr}}{\phi(t)}={-1}.
$$
then there exits $T>0$ such that for all $t>T$ we have 
$$
\left|\frac{\ln{\int_t^{\infty} 
e^{-\phi(r)} dr}}{\phi(t)}+1 \right| \le 1.
$$
In particular 
\begin{equation}\label{estcon}
\int_t^{\infty} e^{-(\phi(r)-\phi(t))} dr \ge e^{-\phi(t)}.
\end{equation}
Using convexity of $\phi$ we get that $\phi(r)-\phi(t)\ge \phi'(t)(r-t)$ and thus, combining this with (\ref{estcon}) we get that
\begin{equation}\label{eq:sd}
\phi'(t)\le e^{\phi(t)}, \mbox{ for all } t>T. 
\end{equation}
Let us show that there is an increasing,
positive, convex, piecewise quadratic  function $\phi$ which has sufficiently large derivative at  a  sequence
of points $t_k \to \infty,$ such $\phi$  would contradict (\ref{eq:sd}).   

We  define function $\phi$ to be  quadratic on  each interval $[k, k+1]$ and show that there exist   $t_k\in (k, k+1)$,  for all $k\in \{0,1, \dots\},$ which would contradict (\ref{eq:sd}).  Let $\phi(0)=\phi'(0)=1$. Assume we have constructed desired  function $\phi$ on interval $[0, k]$ with $\phi(k)=a_k, \phi'(k)=b_k$. Consider an auxiliary quadratic function  $\phi_k: [k, \infty) \to [a_k, \infty)$, such that   $\phi_k'(t)=\alpha_k (t-k)+b_k,$ where $\alpha_k>0$ to be selected later. Thus, $\phi_k(t)=\alpha_k (t-k)^2/2+b_k(t-k)+a_k$. Our  goal is to find $t_k \in (k, k+1)$ and $\alpha_k$ such that 
$
\alpha_k (t-k)+b_k>e^{\alpha_k (t-k)^2/2+b_k(t-k)+a_k}.
$
Let $t_k=k+1/\sqrt{\alpha_k}$. Then  the previous inequality becomes
$
\sqrt{\alpha_k}+b_k>e^{1/2+b_k/\sqrt{\alpha_k}+a_k}
$
which is true for all $\alpha_k$ large enough (in particular allows us to guarantee that $t_k \in (k,k+1)$). We now set $\phi(t)=\phi_k(t)$, for $t \in [k, k+1]$ and repeat the process for the interval $[k+1, k+2]$.
\end{remark}

We remind that for two convex, symmetric bodies $K, L \subset {\mathbb R}^n$, we define $r(K, L)=\max\{R>0: RL\subset K\}$.
The next lemma may be seen as a generalization of the classical large deviation principle (see, for example,  Corollary 4.9.3 in \cite{Bogachev}).
\begin{lemma} \label{lm2}
Consider a symmetric body $K \subset {\mathbb R}^n$, then $$\limsup_{t\to\infty}\frac{\ln{\mu((tK)^c)}}{\phi(r(K, B_2^n)t)}=-1.$$
\end{lemma}
\begin{proof}
Let $R=r(K, B_2^n)$, then $(tK)^c\subset (tRB_2^n)^c$. Using Lemma \ref{lm1} we get $$\limsup_{t\to\infty}\frac{\ln{\mu((tK)^c)}}{\phi(tR)}\leq \limsup_{t\to\infty}\frac{\ln{\mu((tRB_2^n)^c)}}{\phi(tR)}=-1.$$ 
To obtain the reverse inequality, we denote by $P$ a plank of width $2R$ which contains $K$. More precisely, using the maximality of $R$ there exists at least two tangent points  $y, -y \in R{\mathbb S}^{n-1} \cap \partial K$, thus  we may consider  $P=\{x \in {\mathbb R}^n: |\langle x, y\rangle | \le R\}$. Next,  $$\limsup_{t\to\infty}\frac{\ln{\mu((tK)^c)}}{\phi(tR)}\geq \limsup_{t\to\infty}\frac{\ln{\mu((tP)^c)}}{\phi(tR)}.$$
By the rotation invariant of $\mu$, we may assume that $y=Re_n$, and so
$$
\mu((tP)^c)=2\int_{tR}^{\infty} \int_{\mathbb R^{n-1}}e^{-\phi (|ze_{n}+x|)} dx dz.
$$
Using the triangle inequality and the polar coordinates we get
\begin{align*}
    \mu((tP)^c) &\geq  2\int_{tR}^{\infty} \int_{\mathbb R^{n-1}}e^{-\phi (z+|x|)} dx dz \\ 
 &=2\int_{tR}^{\infty} \int_{\mathbb S^{n-2}} \int_0 ^{\infty} e^{-\phi (z+r)} r^{n-2} dr d\theta dz\\
 &=2|{\mathbb S^{n-2}}| \int_0^{\infty}r^{n-2}\int_{tR}^{\infty}e^{-\phi (z+r)} dz dr\\&=2|{\mathbb S^{n-2}}| \int_0^{\infty}r^{n-2}\int_{tR+r}^{\infty}e^{-\phi (z)} dz dr\\ 
&=2|\mathbb S^{n-2}|\int_{tR}^{\infty}e^{-\phi(z)} \int_0^{z-tR} r^{n-2} dr dz\\&=2\frac{|\mathbb{S}^{n-2}|}{n-1} \int_{tR}^{\infty} (z-tR)^{n-1}e^{-\phi(z)} dz.
\end{align*}
Now to finish the proof of Lemma 2 we need to prove the following claim.
\begin{claim}\label{claim2}
   $$
   \limsup_{t\to\infty} \frac{\ln{\int_{tR}^{\infty} (z-tR)^m e^{-\phi(z)} dz} }{\phi (Rt)} \geq -1,
   $$ 
   for any non-negative integer $m$.
\end{claim}
\noindent{\it Proof of Claim \ref{claim2}.}
Making the change of variables we get
$$
   \limsup_{t\to\infty} \frac{\ln{\int_{tR}^{\infty} (z-tR)^m e^{-\phi(z)} dz} }{\phi (Rt)} =
   \limsup_{t\to\infty} \frac{\ln{\int_{t}^{\infty} (r-t)^m e^{-\phi(r)} dr} }{\phi (t)}.
   $$ 
We will first prove the following inductive step: fix a non-negative integer $m $,  and let 
$$F_m(t)=\int_{t}^{\infty} (r-t)^{m}e^{-\phi(r)}dr,$$
then
\begin{equation}\label{induct1}
\liminf_{t\to\infty} \frac{\ln F_m(t)}{\ln F_{m-1}(t)} =1, \mbox{ for all } m\in {\mathbb N}.
\end{equation}

We note that $F_m(t)\le 1$, for $t$ large enough, and thus the denominator and numerator are negative. It is a bit easier to work with a fraction when both the denominator and numerator are non-negative. So we will prove that 
$\liminf\limits_{t\to\infty} \frac{-\ln F_m(t)}{-\ln F_{m-1}(t)} =1.$ 
Using integration by parts, we get
$$
F_{m-1}(t)
=\frac{1}{m}\int_t^\infty (r-t)^m \phi'(r) e^{-\phi(r)} dr \ge  \frac{1}{m}\phi'(t)\int_t^\infty (r-t)^m  e^{-\phi(r)} dr,
$$
where, again, we denote by $\phi'(t)$ the left derivative of $\phi$. Thus
$$  -\ln F_{m-1}(t) \le -\ln(\phi'(t)/m)-\ln F_m(t),
$$
and
$$
\liminf_{t\to\infty} \frac{-\ln F_m(t) }{-\ln F_{m-1}(t) } \ge 
\liminf_{t\to\infty} \frac{\ln(\phi'(t)/m)-\ln F_{m-1}(t)}{
-\ln F_{m-1}(t)}.
$$
Now we may use that  $\phi'(t) >a>0$ for $t$ large enough, and $\lim\limits_{t \to \infty}\ln F_{m-1}(t)=-\infty$ to claim that
\begin{equation}\label{big}
\liminf_{t\to\infty} \frac{-\ln F_m(t) }{-\ln F_{m-1}(t) } \ge 
\liminf_{t\to\infty} \frac{\ln(a/m)-\ln F_{m-1}(t)}{
-\ln F_{m-1}(t)} \ge 1.
\end{equation}

To prove the reverse inequality we note that 
$$ F_m'(t)=-m\int_{t}^{\infty} (r-t)^{m-1} e^{-\phi(r)} dr.
$$
Assume that
$$
\liminf_{t\to\infty} \frac{ -\ln{F_m(t)}}{-\ln(-\frac{1}{m}F_m'(t))} >\alpha>1,
$$
but then, again there exists $t_0>0$ such that for all $t > t_0,$ we have
$$
-\ln{F_m(t)} > -\alpha \ln(-\frac{1}{m}F_m'(t)),
$$
and thus
$$
m < \frac{-F_m'(t)}{F_m(t)^\frac{1}{\alpha}}.
$$
Take an integral over $t \in [x, \infty)$ from both sides to get
$
F_m^{1-\frac{1}{\alpha}}(x) =\infty,
$
which is a contradiction. This finishes the proof of the inductive step, but we
actually need a bit stronger statement, which is similar to Remark \ref{remarkmeasure}. Indeed, consider any $m\in {\mathbb N}$  and $\alpha> 1$. Let
$$
E_{m,\alpha}=\{t: -\ln F_m(t)  >- \alpha \ln F_{m-1}(t)\},
$$
then, using the same ideas as in (\ref{lessinfty}) we get 
 $|E_{m,\alpha}|<\infty$.

To complete our  proof, let  
$$X_i (t)=\frac{\ln F_i(t)}{ \ln F_{i-1}(t)}  \mbox{ and } Y(t)=\frac{\ln F_0(t)}{\phi(t)}.$$ 
Using (\ref{induct1}) we get  
$\liminf\limits_{t\to{\infty}} X_{i}(t)=1$ and using (\ref{inlimsup}) 
 we get $\limsup\limits_{t\to\infty}Y(t)= -1.$ 
 Now let $X(t)=\prod_{i=1}^{m}X_{i}(t)$. Our goal is to prove that 
 $$
\limsup_{t\to\infty}X(t)Y(t) \ge -1.
$$

Assume that this is not true, then there exists $\alpha >1$ such that 
$$
\limsup_{t\to\infty}X(t)Y(t)<{-\alpha}<-1.
$$ Therefore, there exists $t_0$ such that for all $t>t_0$ \begin{equation}\label{assumption}X(t)Y(t)\le-\alpha.
\end{equation}
Using (\ref{big}) we may also assume that $X_i(t)>0$ for all $t>t_0.$ Next, consider the set  $$
A:=\left\{t>t_0: X(t)>\frac{\alpha+1}{2}\right\},$$
we claim that $|A|<\infty$. Note that 
    $$\left|\left\{t:X(t)>\frac{\alpha+1}{2}\right\}\right| \le \left|\left\{t:  X_i(t) >\left(\frac{\alpha+1}{2}\right)^\frac{1}{m}\!\!\!\!\!\!, \mbox{ for some } i\in\{1,\dots,m\}\right\}\right| $$
    $$< \sum_{i=1}^{m} \left|\left\{t:X_{i}(t)>\left(\frac{\alpha+1}{2}\right)^\frac{1}{m} \right\}\right|<\infty.$$
We also note that $\frac{2\alpha}{\alpha+1}>1$ and thus
$$
\left|\left\{t: Y(t) < -\frac{2 \alpha}{\alpha+1}\right\}\right|<\infty.
$$
Finally
$$
\left|\left\{t> t_0: X(t)Y(t) < -\alpha\right\}\right| =\left|\left\{t> t_0: Y(t) < -\frac{\alpha}{X(t)}\right\}\right| $$
$$\le |A|+
\left|\left\{t> t_0: Y(t) < -\frac{\alpha}{X(t)} \mbox{  and } X(t)<\frac{\alpha+1}{2}\right\}\right|$$
$$
\le |A|+
\left|\left\{t: Y(t) < -\frac{2 \alpha}{\alpha+1}\right\}\right|<\infty,
$$
which contradicts with (\ref{assumption}). The claim is proved and this finishes the proof of Lemma \ref{lm2}.
\end{proof}

We are now ready to prove Theorem \ref{thm1}.
\begin{proof}
 Let $K\subset {\mathbb R}^n $ be a convex,  symmetric body such that $\mu(tRB_2^n) \le \mu (tK)$ holds for for some  fixed $R>0$ and every $t$ large enough, but $RB_2^n \not\subset K.$ Thus, the maximal Euclidean ball in $K$ has radius $rR$, with $r\in (0,1)$. From the assumption it follows that $$\mu((tRB_2^n)^c)\ge\mu((tK)^c),$$ which implies that
 \begin{equation*}
      \frac{\ln{\mu((tRB_2^n)^c)}}{\phi(tR)}\ge\frac{\ln{\mu((tK)^c)}}{\phi(trR)}\frac{\phi(trR)}{\phi(tR)}.
 \end{equation*}
 From the convexity of $\phi$ and  $r \in (0,1)$ we get that
 $$\phi(trR)=\phi(trR+(1-r)0)\le r\phi(tR)+(1-r)\phi(0).$$ 
 Using that $\phi(tR) \to \infty $ we get that there exists $r'\in (0,1)$ and $t_0>0$ such that  $\frac{\phi(trR)}{\phi(tR)}\le r',$ for all $t>t_0$. Thus
  \begin{equation}\label{cont2}
      \frac{\ln{\mu((tRB_2^n)^c)}}{\phi(tR)}\ge r'\frac{\ln{\mu((tK)^c)}}{\phi(trR)},
 \end{equation} 
 for all $t>t_0$. Taking the limit superior, as $t\to \infty$, from both sides of the inequality (\ref{cont2}), we obtain $-1\ge -r'.$ But this contradicts the fact that $r'$ is less than $1$. Therefore, our assumption that $RB_2^n \not\subset K$ must be false.
\end{proof}


\begin{remark} The rotation invariant assumption on $\mu$ in Theorem \ref{thm1} is necessary. Indeed, one can construct an example of a $\log$-concave probability measure that is not rotation invariant in $\mathbb{R}^2$ which does not satisfy the statement of Theorem \ref{thm1}. Consider the rectangle $\Omega=\{(x, y): |x|\le\frac{\pi}{2}, |y|\le \frac{1}{2}\}$, and define the measure $\mu$ as  $\mu (K)=\frac{|K\cap \Omega|}{|\Omega|}$.
Taking $K=\Omega$, we have  
$$
B_2^2\not\subset \Omega, \mbox{ but  }  |B_2^2|=|\Omega|=\pi \mbox{ and } |tB_2^2|=|t\Omega|, \hspace{0.1cm}\forall t>0.
$$
Note that $\mu(tB_2^2)\le\mu(t\Omega);\hspace{.1cm} \forall t>0$, indeed, this is equivalent to $|tB_2^2\cap \Omega|\le |t\Omega\cap \Omega|$. If $t\le 1$, then we have $|tB_2^2\cap \Omega|\le |t B_2^2|=|t\Omega|$, and if $t\ge 1$ we get $|tB_2^2 \cap \Omega|\le |\Omega|.$ So, we provided an example where $$\mu(tB_2^2)\le\mu(tK);\hspace{.5cm} \forall t>0$$ but $B_2^2\not\subset K.$ 
\end{remark}

\section{The cases where density depends on the norm}\label{genL}
    In this section, we would like to give a proof   Theorem \ref{thm1} in a more general case, which would answer Question \ref{ques2}. The main idea and computation are in the same spirit as in the proof of Theorem \ref{thm1}.
   \begin{theorem}  \label{thm2}
        Let $K, L\subset \mathbb{R}^n$ be convex,   symmetric bodies, and let $\mu$ be a $\log$-concave probability measure, with  density 
$e^{-\phi(\|x\|_{L})}$, where $\phi: [0,\infty) \to [0, \infty)$ is an increasing, convex function. If for every $t$ large enough and some $R>0$
$$
\mu(tRL) \le \mu (tK),
$$
then $RL \subseteq K.$
   \end{theorem}
   
   \begin{proof}
       We have to check Lemma \ref{lm1} and Lemma \ref{lm2}, i.e. to prove yet another generalization of the classical large deviation principle (see (\ref{ldL}) below).
       
       We claim that for any  $R>0$ 
    \begin{equation*}
    \limsup_{t\to \infty}\frac{\ln{\mu((tRL)^c)}}{\phi(tR)}=-1.
    \end{equation*}
We can assume $R=1$. Moreover, as before, using convexity of $\phi$ we may assume that $\phi(t)$ is a strictly increasing function for large enough $t$. Thus,
\begin{align}\label{trick}
    \mu(&(tL)^c) = \int_{(tL)^c} e^{-\phi (\|x\|_L)} dx  = \int_{(tL)^c} \int_{\phi(\|x\|_L) }^{\infty} e^{-u} du dx \nonumber \\&= \int_{\mathbb{R}^n} \int_{\phi(\|x\|_L)}^{\infty} \!\!\!\chi_{(tL)^c} (x) e^{-u} du dx= \int_{0}^{\infty} \int_{\{x: \phi(\|x\|_L)<u\}}\!\!\! \chi_{(tL)^c} (x) e^{-u} dx du \nonumber \\&= \!\int_{0}^{\infty}\!\!\!e^{-u} |\{x\!\!:  \|x\|_L \in [t, \phi^{-1}(u)]\}| du = |L| \!\int_{\phi(t)}^{\infty} ((\phi^{-1}(u))^n-t^n) e^{-u}du \nonumber\\&= |L| \int_{t}^{\infty} (v^n-t^n) \phi^{\prime}(v) e^{-\phi(v)} dv = -|L| \int_{t}^{\infty} (v^n-t^n) d e^{-\phi(v)}\nonumber\\& = n|L| \int_{t}^{\infty} v^{n-1} e^{-\phi(v)} dv,
\end{align}
note that  $\phi(t)$ may be a constant function on some interval  $[0, t_0]$ and strictly increasing on $[t_0, \infty)$, in such a case  we define  $\phi^{-1}(\phi(0))=t_0$.
     So, we have \begin{align*}
    \limsup_{t\to\infty}\frac{\ln\mu((tL)^c)}{\phi(t)} &=\limsup_{t\to\infty}\frac{\ln ({n|L|\int_t^{\infty} 
e^{-\phi(v)} v^{n-1}dv})}{\phi(t)} \\&=\limsup_{t\to\infty}\frac{\ln {\int_t^{\infty} 
e^{-\phi(v)} v^{n-1}dv}}{\phi(t)},\\
&=-1,
     \end{align*}
where the last equality follows from the proof of  Lemma \ref{lm1}. 

    To finish the proof, we must check Lemma \ref{lm2}. In particular, we want to show that
    \begin{equation}\label{ldL}\limsup_{t\to\infty}\frac{\ln{\mu((tK)^c)}}{\phi(r(K,L)t)}=-1,
    \end{equation}
for  symmetric, convex bodies $K, L \subset {\mathbb R}^n$, convex, increasing function $\phi:[0, \infty) \to [0, \infty)$ and measure $\mu$ with density $e^{-\phi(\|x\|_L)}$.

Let $R=r(K,L)$, then we have $(tK)^c\subset (tRL)^c$ from the assumption. Thus using Lemma \ref{lm1} we get $$\limsup_{t\to\infty}\frac{\ln{\mu((tK)^c)}}{\phi(tR)}\leq \limsup_{t\to\infty}\frac{\ln{\mu((tRL)^c)}}{\phi(tR)}=-1.$$ 
Using that   $RL$ is the maximal dilate of $L$ inside $K$ we get that there is a pair of points $v, -v \in \partial RL\cap \partial K$. Let $P$ be a plank created by tangent planes to $RL$ and $K$ at $v$ and $-v$. Let $n_v$ be a normal vector to $\partial RL$ at $v$. Then the width of the plank $P$ is $2Rh_L(n_v)=2h_K(n_v)$: $P=\{x \in {\mathbb R}^n: |\langle x, n_v\rangle| \le Rh_L(n_v)\}$, where $h_L (x)=\sup \{\langle x, y \rangle  : y\in L\}$ is the support function of $L$ (see \cite{Schneider} for basic definitions and properties). 
Next,  $$\limsup_{t\to\infty}\frac{\ln{\mu((tK)^c)}}{\phi(tR)}\geq \limsup_{t\to\infty}\frac{\ln{\mu((tP)^c)}}{\phi(tR)}.$$
Selecting a proper system of coordinates, we may assume that $n_v=e_n$, let $a=tRh_L(e_n).$ Then,
\begin{align*}
    \mu((tP)^c) &= 2\int_{a}^{\infty} \int_{e_n^\perp}e^{-\phi (\|ze_n +x\|_L)} dx dz\\ &= 2\int_{a}^{\infty}\int_{e_n^\perp} \int^\infty_{\phi (\|z e_n + x\|_L)} e^{-u} du dx dz \\
    &= 2\int_{a}^{\infty}\int_{0}^\infty \int_{\{x \in e_n^\perp: \phi (\|ze_n+ x\|_L)<u\}} e^{-u} dx du dz  \\
    &= 2\int_{a}^{\infty}\int_{0}^\infty e^{-u}\left|\left\{x \in e_n^\perp: \|ze_n + x\|_L\le \phi^{-1}(u)\right\}\right| du dz.
\end{align*}
 Now note that 
$$
\left|\left\{x \in e_n^\perp: \|ze_n + x\|_L\le \phi^{-1}(u)\right\}\right|=
\left|\left\{x \in e_n^\perp: ze_n + x \in  \phi^{-1}(u)L\right\} \right|.
$$
The above volume is zero if $z> \phi^{-1}(u)h_L(e_n)$ (or $\phi(z/h_L(e_n)) >u$), for $z \in [0, \phi^{-1}(u)h_L(e_n)]$ we note that $\phi^{-1}(u) L$ is a convex body and thus contains inside a pyramid $\Delta$ with base $L\cap e_n^\perp$ and the height $\phi^{-1}(u)h_L(e_n)$ (with apex $\phi^{-1}(u)h_L(e_n)v/R$), then
$$
\left|\left\{x \in e_n^\perp: ze_n + x \in  \phi^{-1}(u)L \right\}\right| \ge |\Delta\cap (e_n^\perp+ze_n)|$$
$$=(\phi^{-1}(u)h_L(e_n)-z)^{n-1}|L\cap e_n^\perp|.
$$
Thus
\begin{align*}
    \mu((tP)^c) &\ge 2 |L\cap e_n^\perp|\int_{a}^{\infty}\int_{\phi(z/h_L(e_n))}^\infty e^{-u}
    (\phi^{-1}(u)h_L(e_n)-z)^{n-1}
    du dz \\
    &= 2 |L\cap e_n^\perp|\int_{a}^{\infty}\int_{z/h_L(e_n)}^\infty e^{-\phi(u)}\phi'(u)
    (u h_L(e_n)-z)^{n-1}
    du dz\\
    &= -2 |L\cap e_n^\perp|\int_{a}^{\infty}\int_{z/h_L(e_n)}^\infty 
    (u h_L(e_n)-z)^{n-1}
    de^{-\phi(u)} dz\\
    &= 2(n-1) |L\cap e_n^\perp|\int_{a}^{\infty}\int_{z/h_L(e_n)}^\infty e^{-\phi(u)}
    (u h_L(e_n)-z)^{n-2}
    du dz\\
     &= 2(n-1) |L\cap e_n^\perp|\int_{a/h_L(e_n)}^{\infty}\int_{a}^{h_L(e_n)u} e^{-\phi(u)}
    (u h_L(e_n)-z)^{n-2}
    dz du\\
    &= 2 |L\cap e_n^\perp|\int_{a/h_L(e_n)}^{\infty} e^{-\phi(u)}
    (u h_L(e_n)-a)^{n-1}
    du\\
    &= 2 h_L^{n-1}(e_n)|L\cap e_n^\perp|\int_{tR}^{\infty} e^{-\phi(u)}
    (u -tR)^{n-1}
    du.\\
\end{align*}
So, we have \begin{align*}
    \limsup_{t\to\infty}\frac{\ln \mu((tP)^c)}{\phi(tR)} &\ge \limsup_{t\to\infty}\frac{\ln \left(2 h_L^{n-1}(e_n)|L\cap e_n^\perp|\int_{tR}^{\infty} e^{-\phi(u)}
    (u -tR)^{n-1} du\right)}{\phi(tR)} \\&=\limsup_{t\to\infty}\frac{\ln\int_{tR}^{\infty} e^{-\phi(u)} (u -tR)^{n-1}du }{\phi(tR)}.
     \end{align*}
    By Claim 2 the above quantity is greater than or equal to $-1$, thus Lemma \ref{lm2} is applied here which finishes the proof for our main result. 
   \end{proof}

\begin{remark}
    The proofs for Theorem \ref{thm1} and Theorem \ref{thm2} apply similarly to an asymmetric convex body $K$ with the origin as an interior point of it, the only difference is that instead of dealing with a plank $P$ in Lemma \ref{lm2} we need to work with a half-space. Specifically, for Theorem \ref{thm1}, one would use  the half-space $H=\{x \in {\mathbb R}^n: \langle x, y\rangle  \le R\}$ where $y \in R{\mathbb S}^{n-1} \cap \partial K$. For Theorem \ref{thm2}, one may  use the half-space defined by $H=\{x \in {\mathbb R}^n: \langle x, n_v\rangle \le Rh_L(n_v)\},$ where $n_v$ is the normal vector to $\partial RL$ at a tangent point $v$.
\end{remark}
    
\section{The Busemann - Petty type problems}\label{BPdilation}
   
In this section, we will discuss Question \ref{BPr}. We first note that one must make some additional assumptions on the measure $\mu$ to avoid a trivial answer. Indeed, if a measure $\mu$ has a homogeneous density (i.e. $f(rx)=r^pf(x),$ for $r>0$ and $p>1-n$), then the answer is identical to the one given in \cite{Zvavitch2}. 

Let us first show that in dimension $n\ge 5$ one can always find a pair of convex, symmetric bodies $K$ and $L$ and measure $\mu$, such that the answer to Question \ref{BPr} is negative. The main idea follows from the construction in  \cite{Zvavitch3}. We begin with the following fact:

\noindent{\bf Fact.} {\it  If $d\mu= e^{-\phi (\|x\|_L)} dx$ is a $\log$-concave  measure and $K, L\subset \mathbb{R}^n$ are convex,  symmetric bodies such that $|K|\le|RL|$ for some $R>0,$ then $$ \mu(K)\le \mu(RL).$$}

\begin{proof}
    Using calculations similar to (\ref{trick}) we get 
    \begin{equation*}
    \mu(K) = \int_{0}^{\infty} e^{-u} |K\cap\phi^{-1}(u) L| du
\end{equation*}
and
$$
\mu(RL)= \int_{0}^{\infty} e^{-u} |RL\cap\phi^{-1}(u) L| du.
$$ 
 To get $ \mu(K)\le \mu(RL)$ we only need  to check that $|K\cap\phi^{-1}(u) L|\le |RL\cap\phi^{-1}(u) L|.$ Indeed, if $R\le \phi^{-1}(u),$ then we have $$|K\cap\phi^{-1}(u) L|\le |K|\le|RL|= |RL\cap\phi^{-1}(u) L|,$$ and if $R\ge \phi^{-1}(u),$ then we get $$|K\cap\phi^{-1}(u) L|\le |\phi^{-1}(u) L|=|RL\cap\phi^{-1}(u) L|.$$ Hence, $ \mu(K)\le \mu(RL),$ for any $R>0.$
\end{proof}

Next, we show that Question \ref{BPr} has a negative answer for $n\ge 5$.
\begin{theorem} \label{thm3}
    For $n\ge 5,$ there are convex symmetric bodies $K, L\subset \mathbb{R}^n$  and 
     $\log$-concave measure $\mu$  with density $e^{-\phi(\|x\|_L)}$, such that 
    \begin{equation} \label{D2}
    \mu (rK\cap \xi^\perp)\le \mu (rL\cap \xi^\perp),  \hspace{0.5cm} \forall \xi\in\mathbb{S}^{n-1}, \hspace{0.5cm}\forall r>0,
        \end{equation}
        but $\mu(K) > \mu(L).$
\end{theorem}
\begin{proof}
    Let us assume, towards the contradiction,  that Question \ref{BPr} has an affirmative answer in $\mathbb{R}^n$, for some fixed $n\ge 5.$ So, for any pair of convex  symmetric bodies $K, L$ that satisfy (\ref{D2}), we would get $\mu(K)\le \mu(L).$ The condition on sections (\ref{D2}) will be also satisfied for the dilated bodies $tK$ and $tL$, for all $t>0.$ Therefore, we have \begin{equation} \label{D2.1}
        \mu(tK) \le \mu (tL), \hspace{.5cm} \forall t>0.
    \end{equation} Which by definition of $\mu$ means $$\int_{tK}{e^{-\phi(||x||_L)} dx} \le \int_{tL}{e^{-\phi(||x||_L)} dx},$$ or equivalently, applying the change of variables $x = tx$, we have $$\int_{K}{e^{-\phi(t||x||_L)} dx} \le \int_{L}{e^{-\phi(t||x||_L)} dx}.$$ Using the continuity of $\phi$ and compactness of $K$ and $L$, we can take the limit for the above inequality as $t \to 0^+$, to obtain $$|K|\le|L|.$$
    Therefore, we have a relation between the dilation problem for a $\log$-concave probability measure, with the Busemann-Petty problem for volume measure, which is if $$ \mu (rK\cap \xi^\perp)\le \mu (rL\cap \xi^\perp)),  \hspace{0.5cm} \forall \xi\in\mathbb{S}^{n-1}, \hspace{0.5cm}\forall r>0,$$ then $|K|\le|L|$.

A number of very interesting counterexamples to the Busemann-Petty problem were shown by Papadimitrakis \cite{Papadimitrakis-1992};  Gardner  \cite{Gardner-1994};   Gardner,  Koldobsky, and Schlumprecht  \cite{GKS}: there are convex symmetric  bodies $K, L$ in $\mathbb{R}^n$ for $n\ge 5$ such that
    \begin{equation} \label{GKS}
        |K \cap\xi ^\perp|\le|L \cap \xi ^\perp|, \hspace{.5cm} \forall \xi\in\mathbb{S}^{n-1},
    \end{equation} but 
    \begin{equation}\label{Gia}
    |K|> |L|.
    \end{equation}
    Note that because the volume measure is homogeneous the condition on sections (\ref{GKS}) is also true for dilates of $K$ and $L$, so we have \begin{equation} \label{GKS1}
        |rK \cap\xi ^\perp|\le|rL \cap \xi ^\perp|, \hspace{.5cm} \forall \xi\in\mathbb{S}^{n-1},\hspace{0.5cm}\forall r>0.
    \end{equation}
Now, applying the Fact  to (\ref{GKS1}), we get that 
$$\mu (rK\cap \xi^\perp)\le \mu (rL\cap \xi^\perp),  \hspace{0.5cm} \forall \xi\in\mathbb{S}^{n-1}, \hspace{0.5cm}\forall r>0.$$
Thus,   using (\ref{D2.1}), we have $$\mu(tK) \le \mu (tL), \hspace{.5cm} \forall t>0, $$
dividing by $t^n$ and taking the limit of the above inequality as $t \to 0^+$ we get $$|K| \le |L|,$$ and this contradicts (\ref{Gia}).  
    
\end{proof}
It is interesting to note that the measure $\mu$ constructed above is very specific. For example, we can not use this construction directly with the assumption that $\mu$ is rotation invariant. 

Still, we can show that the answer to  Question \ref{BPr} is negative in ${\mathbb R}^n$ for $n\ge 7$ even when $\mu$ is a $\log$-concave  measure with rotation invariant density.
\begin{theorem} \label{thm4}
    For dimension $n\ge 7,$ and $d\mu=e^{-\phi(|x|)} dx$ there is a convex  symmetric body $K\subset \mathbb{R}^n$ such that 
    
  $$  \mu (rK\cap \xi^\perp)\le \mu (rB_2^n\cap \xi^\perp),  \hspace{0.5cm} \forall \xi\in\mathbb{S}^{n-1}, \hspace{0.5cm}\forall r>0,$$
        but $\mu(K) > \mu(B_2^n).$
\end{theorem}
\begin{proof}
Giannopoulos \cite{Giannopoulos-1990}  and  Bourgain \cite{Bourgain1} constructed an example in $\mathbb{R}^n$ for $n\ge 7$ of convex body $K\subset \mathbb{R}^n$, that satisfies 
       $$ |K \cap\xi ^\perp|\le|B_2^n \cap \xi ^\perp|, \hspace{.5cm} \forall \xi\in\mathbb{S}^{n-1},$$
     but $|K|> |B_2^n|$. To prove Theorem  \ref{thm4} one may take the same convex body $K$ and $B_2^n$ as provided in \cite{Giannopoulos-1990, Bourgain1} and repeat the proof of Theorem  \ref{thm3}.   
\end{proof}

\vspace{3mm}

\noindent {\sc Department of Mathematical Sciences, Kent State University, Kent, OH USA.}

\noindent {\it E-mail address:} {\tt mlafi1@kent.edu, zvavitch@math.kent.edu}

\end{document}